\documentclass[letterpaper, 10 pt, conference]{ieeeconf}  

\IEEEoverridecommandlockouts                              

\usepackage{graphics} 
\usepackage{epsfig} 
\usepackage{times} 
\usepackage{amsmath} 
\usepackage{amssymb}  
\usepackage{amsfonts}
\newtheorem{thm}{Theorem}

\newtheorem{remark}{Remark}
\newtheorem{definition}{Definition}
\newtheorem{proposition}{Proposition}
\newtheorem{assumption}{Assumption}
\usepackage{color}
\usepackage{algorithm}
\usepackage{algpseudocode}
\usepackage{lipsum}
\usepackage{xcolor}

\usepackage{tikz}
\usetikzlibrary{shapes.geometric, 
                patterns.meta,
                arrows.meta,
                intersections,
                shapes.arrows,
                decorations.markings,
                fadings, 
                shadows.blur,
                hobby,
                3d,
                backgrounds}

\definecolor{MariaBlue}{RGB}{74, 144, 226}
\definecolor{MariaGreen}{RGB}{80, 227, 194}

\title{\LARGE \bf
Optimal Control of Reduced Left-Invariant Hybrid Control Systems*
}

\author{William Clark$^{1}$ and Maria Oprea$^{2}$
\thanks{*This work was funded by AFOSR grant FA9550-23-1-0400}
\thanks{$^{1}$William Clark is with the Department of Mathematics, Ohio University, Athens, OH 45701, USA
        {\tt\small clarkw3@ohio.edu}}%
\thanks{$^{2}$Maria Oprea is with the Center for Applied Mathematics, Cornell University, Ithaca, NY 14853, USA
        {\tt\small mao237@cornell.edu}}%
}

\begin{document}

\maketitle
\thispagestyle{empty}
\pagestyle{empty}

\begin{abstract}
Optimal control is ubiquitous in many fields of engineering. A common technique to find candidate solutions is via Pontryagin's maximum principle. An unfortunate aspect of this method is that the dimension of system doubles. When the system evolves on a Lie group and the system is invariant under left (or right) translations, Lie-Poisson reduction can be applied to eliminate half of the dimensions (and returning the dimension of the problem to the back to the original number). 

Hybrid control systems are an extension of (continuous) control systems by allowing for sudden changes to the state. Examples of such systems include the bouncing ball - the velocity instantaneously jumps during a bounce, the thermostat - controls switch to on or off, and a sailboat undergoing tacking. The goal of this work is to extend the idea of Lie-Poisson reduction to the optimal control of these systems. If $n$ is the dimension of the original system, $2n$ is the dimension of the system produced by the maximum principle. In the case of classical Lie-Poisson reduction, the dimension drops back down to $n$. This, unfortunately, is impossible in hybrid systems as there must be an auxiliary variable encoding whether or not an event occurs. As such, the analogous hybrid Lie-Poisson reduction results in a $n+1$ dimensional system. The purpose of this work is to develop and present this technique.
\end{abstract}

\section{INTRODUCTION}
Many systems naturally evolve on Lie groups, e.g. a ship can be described by a point in $\mathrm{SE}_2$, a satellite by $\mathrm{SO}_3$, the double pendulum on $\mathbb{T}^2$, etc. As the underlying state-space is endowed with a group structure, the control law can exploit these symmetries by being left-invariant. If $G$ is the Lie group and $\mathcal{U}$ is the control set, then a left-invariant controlled vector field is a map $f:G\times\mathcal{U}\to TG$ such that
\begin{equation}\label{eq:intro_left_inv}
    \dot{g} = f(g,u) = g\cdot\tilde{f}(u), \quad \tilde{f}:\mathcal{U}\to\mathfrak{g},
\end{equation}
where $\mathfrak{g} = \mathrm{Lie}(G)$ is the corresponding Lie algebra.

A common question to attach to the control system \eqref{eq:intro_left_inv} is that of optimal control, i.e. find a control $u:[t_0,t_f]\to\mathcal{U}$ that minimizes some cost,
\begin{equation*}
    J\left(u(\cdot)\right) = \int_{t_0}^{t_f} \, L\left(g(s),u(s)\right) \, ds.
\end{equation*}
When the cost is also left-invariant, the resulting optimal control problem is left-invariant and has been extensively studied, e.g. \cite{brockett, JURDJEVIC1972313, Biggs2017, COLOMBO2023131, Sachkov_2022}.

A common paradigm for solving optimal control problems is via Pontraygin's maximum principle. This technique transforms the $n=\mathrm{dim}(G)$-dimensional control system into a split boundary value problem of dimension $2n$. When the underlying control system is left-invariant, Lie-Poisson reduction can be applied to reduce the dimension back to $n$. To aim of this work is to extend this reduction technique to hybrid control systems as shown in Fig. \ref{fig:intro_figure}, i.e. replacing the dynamics \eqref{eq:intro_left_inv} by
\begin{equation*}
    \begin{cases}
        \dot{g} = f(g,u), & g\not\in \Sigma, \\
        g^+ = \Delta(g), & g\in \Sigma.
    \end{cases}
\end{equation*}
In order to apply (hybrid) Lie-Poisson reduction to this problem, in addition to $f$ being left-invariant, both $\Sigma$ and $\Delta$ need to satisfy analogous properties. As will be shown, the presence of state jumps limits the strength of the reduction and only permits the dimension to be reduced to $n+1$.

\begin{figure}
    \centering
    \input{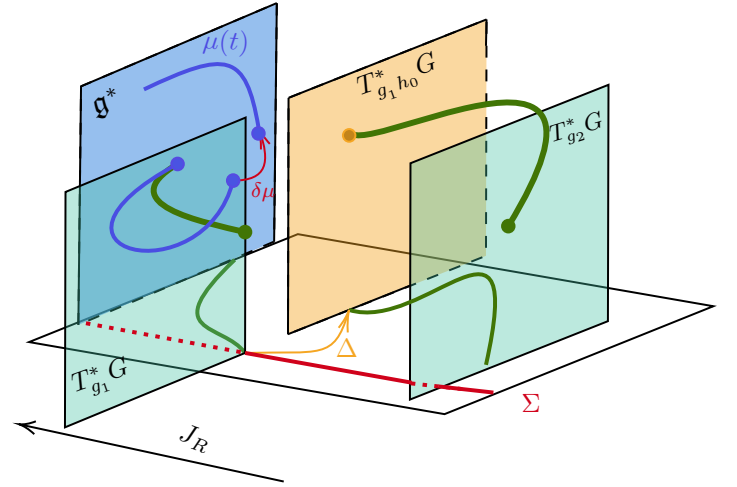}
    \caption{A schematic of Lie-Poisson reduction as applied to hybrid systems. An event occurs at $g_1\in\Sigma$ and the arc gets mapped to $\Delta(g_1)$. The map $J_R:T^*G\to\mathfrak{g}^*$ allows for the dimension reduction, and the corresponding reduced momentum jumps by $\delta \mu$}
    \label{fig:intro_figure}
\end{figure}
The structure of this paper is the following: Section \ref{sec:LIHS} defines the notion of a left-invariant hybrid control system on a Lie group. Section \ref{sec:LIOC} presents an overview of the maximum principle as it applied to (purely continuous) left-invariant optimal control problems and its relation wot Lie-Poisson reduction. Section \ref{sec:LIHOC} extends the techniques of left-invariant optimal control to hybrid control systems and contains the main results of this work: Theorem \ref{thm:main_result} and Algorithm \ref{alg:hybrid_LPO}. Section \ref{sec:SE2} applies the theory to $\mathrm{SE}_2$, the special Euclidean group of the plane. This work concludes with section \ref{sec:conclusion}.
\section{LEFT-INVARIANT HYBRID SYSTEMS}\label{sec:LIHS}
When the state-space for a control system is a Lie group, it is natural to examine problems that are left-invariant to fully exploit the group structure present. A substantial amount can be understood of these systems such as controllability and optimization \cite{left_inv_control}. Left-invariant control systems on a matrix group has dynamics of the following form
\begin{equation*}
    \dot{g} = g\cdot f(u),
\end{equation*}
where $f$ is the manifestation of the controls. To extend this idea to hybrid systems a guard and reset need to be introduced, along with the correct notion of ``left-invariant.'' 

Before specializing to the group case, we first present the general notion of a controlled hybrid system. A controlled hybrid dynamical system has the form
$    \mathcal{HC} = \left(M, \Sigma, \mathcal{U}, f, \Delta\right),
$ where
\begin{enumerate}
    \item $M$ is a finite-dimensional manifold, called the \textit{state-space},
    \item $\Sigma\subset M$ is a co-dimension 1 embedded submanifold, called the \textit{guard},
    \item $\mathcal{U}\subset\mathbb{R}^m$ is a closed set of \textit{admissible controls},
    \item $f:\mathcal{U}\times M\to TM$ is the \textit{controlled vector field}, and
    \item $\Delta:\Sigma\to M$ is a smooth map, called the \textit{reset}.
\end{enumerate}
In the case where the state-space is a Lie group, additional structure may be placed on the control system. Let $G$ be a finite-dimensional Lie group with Lie algebra $\mathfrak{g} = T_eG$. For a given element, $g\in G$, let $\ell_g:G\to G$ be left-translation and $r_g:G\to G$ be right-translation
\begin{equation*}
    \ell_g(h) = gh, \quad r_g(h) = hg,
\end{equation*}
which are diffeomorphisms. Denote its derivative by $(\ell_g)_*:TG\to TG$, and the pullback $r_g^*:T^*G \to T^*G $ by $J_R$ (see Fig. \ref{fig:intro_figure}). In the case where the group is a matrix group, $(\ell_g)_*$ is simply matrix multiplication. At first glance, a hybrid control system on a Lie group is left-invariant if it is invariant under the map $\ell_g$. A key piece of this invariance is that the guard must also be left-invariant. However, this cannot be true unless $\Sigma = G$, the whole group. Rather, the correct notion is that of being tangent preserving.
\begin{definition}[Tangent Preserving \cite{oprea2023study}]
    Let $\Sigma\subset G$ be a submanifold (not necessarily a subgroup). $\Sigma$ is (left) tangent preserving if for any $\sigma\in\Sigma$ and $g\in G$ such that $g\sigma\in\Sigma$, we have $(\ell_g)_*T_\sigma\Sigma = T_{g\sigma}\Sigma.$
\end{definition}

\begin{definition}[Left-Invariant Hybrid Control System]
    A left-invariant hybrid control system (LIHCS) is a refinement of a controlled hybrid dynamical system with $\mathcal{LHC} = (G,\Sigma,\mathcal{U}, f, \Delta)$ where
    \begin{enumerate}
        \item $G$ is a finite-dimensional Lie group,
        \item $\Sigma\subset G$ is tangent-preserving,
        \item $f:\mathcal{U}\to \mathfrak{g}$,
        \item $\Delta:\Sigma\to G$ satisfies
        \begin{equation}\label{eq:equivariant}
            \Delta(g\alpha) = g\Delta(\alpha), \quad \alpha,g\alpha\in \Sigma.
        \end{equation}
    \end{enumerate}
\end{definition}

The required ``equivariance'' property of $\Delta$ does not hold for every $h\in G$; it only holds for elements in the set
\begin{equation*}
    G_\alpha(\Sigma) := \left\{ g\in G :g\alpha\in\Sigma\right\}.
\end{equation*}
Likewise to be tangent preserving, one only needs to check against elements in this set. To elucidate the properties of these two objects, we have the following proposition.
\begin{proposition}\label{prop:coset}
    Let $K\leq G$ be a (closed) co-dimension 1 subgroup. For arbitrary $g_0,h_0\in G$, the right coset $\Sigma:= Kg_0$ is tangent preserving and the map
    \begin{equation*}
        \Delta(\alpha) = r_{h_0}(\alpha),
    \end{equation*}
    satisfies $\Delta(g\alpha) = g\Delta(\alpha)$ for all $g\in G_\alpha(\Sigma)$.
\end{proposition}
\begin{proof}
    Let $\alpha\in\Sigma = Kg_0$ and $g\in G_\alpha(\Sigma)$. As shown in \cite{oprea2023study}, $G_\alpha(\Sigma)=K$ and $\ell_g:\Sigma\to\Sigma$ is a diffeomorphism. Differentiating this shows the tangent preserving property.

    For the reset map, notice that
    \begin{equation*}
        r_{h_0}(g\alpha) = g\alpha h_0 = gr_{h_0}(\alpha).
    \end{equation*}
    Hence, $\Sigma=Kg_0$ is tangent preserving and $\Delta = r_{h_0}$ satisfies \eqref{eq:equivariant}.
\end{proof}
Cosets form a partition of the group via the equivalence relation
\begin{equation*}
    x\sim y \iff xy^{-1}\in K.
\end{equation*}
Equivalence classes of this relation will be denoted by
\begin{equation*}
    [x] = \left\{ y\in G : y\sim x\right\}.
\end{equation*}
In particular, if $\pi:G\to K\backslash G$ is the natural projection,
\begin{equation*}
    [x] = \left\{ y\in G : \pi(y) = \pi(x)\right\}.
\end{equation*}

Throughout the remainder of this work, the guard and reset maps of left-invariant hybrid control systems will have the form specified in Proposition \ref{prop:coset}.
However, there is a slightly more distinguished case that will be important to the task of reduction.
\begin{definition}[Normal LIHCS]
    A left-invariant hybrid control system is normal (NLIHCS) if $\Sigma = Kg_0$ for some (closed, co-dimension 1) normal subgroup $K\trianglelefteq G$.
\end{definition}

The dynamics of a left-invariant hybrid control system can be described via
\begin{equation}\label{eq:hybrid_LICS}
    \mathcal{LHC}:\begin{cases}
        \dot{g} = \left(\ell_g\right)_*f(u), & g\not\in Kg_0, \\
        g^+ = r_{h_0}(g^-), & g\in Kg_0.
    \end{cases}
\end{equation}
In order to guarantee that reset events are isolated (and thus Zeno is prohibited), we wish to have $\overline{\Delta(\Sigma)}\cap\Sigma = \emptyset$. This happens as long as $Kg_0 \ne Kg_0h_0$ (recall that cosets partition the group). This is enforced via the following assumption.
\begin{assumption}\label{ass:non_trivial_jump}
    We will henceforth assume that $g_0h_0g_0^{-1}\not\in K$. When $K$ is normal, this is equivalent to $h_0\not\in K$. This guarantees that $Kg_0\ne Kg_0h_0$.
\end{assumption}

\begin{remark}
    In the case where $G$ is a matrix group and the controls are affine, the controlled dynamics become
    \begin{equation*}
        \dot{g} = g\left( e^0 + u_1e^1 + \ldots + u_k e^k\right),
    \end{equation*}
    where $\{e^j\}_{j=0}^k$ is a collection of elements in $\mathfrak{g}$. In this case, an affine left-invarint hybrid control system has dynamics
    \begin{equation*}
        \begin{cases}
            \dot{g} = g\left( e^0 + u_1e^1 + \ldots u_ke^k\right), & gg_0^{-1}\not\in K, \\
            g^+ = g^-h_0, & gg_0^{-1}\in K.
        \end{cases}
    \end{equation*}
\end{remark}
Fundamental control questions such as controllability, accessibility, observability, etc., of the control system \eqref{eq:hybrid_LICS} are not explored here and are the subject of future work.
\subsection{Example: Special Euclidean Group}
As a simple demonstration of a normal left-invariant hybrid control system, let $G = \mathrm{SE}_2$ be the special Euclidean group on the plane. An element of this group can be described by
\begin{equation*}
    \mathrm{SE}_2 = \left\{ \begin{pmatrix}
        \cos\theta & \sin\theta & x \\
        -\sin\theta & \cos\theta & y \\
        0 & 0 & 1
    \end{pmatrix} \right\}.
\end{equation*}
The collection of translations form a normal subgroup and a right coset has the form
\begin{equation*}
    \Sigma = Kg_0 = \left\{ \left.\begin{pmatrix}
        \cos\theta^* & \sin\theta^* & x \\
        -\sin\theta^* & \cos\theta^* & y \\
        0 & 0 & 1
    \end{pmatrix}\right|\ x,y\in\mathbb{R}\right\},
\end{equation*}
where $\theta^*\in\mathbb{S}^1$ is some critical angle. Assuming fully actuated control, the resulting normal left-invariant hybrid control system has the continuous dynamics
\begin{equation}\label{eq:continuous_se2}
    \begin{split}
        \dot{x} &= u\cos\theta + v\sin\theta, \\
        \dot{y} &= v\cos\theta - u\sin\theta, \\
        \dot{\theta} &= \omega.
    \end{split}
\end{equation}
A reset occurs when $\theta = \theta^*$ and jumps according to
\begin{equation}\label{eq:impact_se2}
    \begin{split}
        x &\mapsto x + \tilde{x}\cos\theta + \tilde{y}\sin\theta, \\
        y &\mapsto y + \tilde{y}\cos\theta - \tilde{x}\sin\theta, \\
        \theta &\mapsto \theta + \tilde{\theta},
    \end{split}
\end{equation}
where $\tilde{x}$, $\tilde{y}$, and $\tilde{\theta}$ are arbitrary. Assumption \ref{ass:non_trivial_jump} states that $\tilde{\theta}\ne 0, \pm 2\pi, \pm4\pi, \ldots$
\section{LEFT-INVARIANT OPTIMAL CONTROL}\label{sec:LIOC}
A commonly used necessary condition in optimal control is Pontryagin's maximum principle. This method lifts the control problem to the co-tangent bundle of the state-space and the flow of a distinguished Hamilton produces the optimal trajectories. In the case of a left-invariant control system with left-invariant cost, the constructed Hamiltonian is also left-invariant. This allows for Lie-Poisson reduction to reduce the dimension of the problem from $2n$ down to $n$. This section presents an overview of the maximum principle, Lie-Poisson reduction and the application of the latter to the former in the context of purely continuous control systems.
\subsection{Pontraygin Maximum Principle}
We begin by demonstrating necessary conditions for the purely continuous optimal control problem of the form
\begin{equation}\label{eq:cost}
    \mathcal{J}\left(x_0\right) = \min_{u(\cdot)}\,\int_{t_0}^{t_f} \, L(x(t),u(t)) \, dt + \varphi(x(t_f)),
\end{equation}
subject to the dynamics
\begin{equation}\label{eq:cont_control}
    \dot{x} = f(x,u).
\end{equation}

Pontraygin's maximum principle \cite{liberzon_oc} states that regular optimal trajectories are intimately related to the flow from the Hamiltonian
\begin{gather*}
    H:T^*M\to \mathbb{R} \\
    H(x,p) = \min_u \, \left[\langle p, f(x,u)\rangle + L(x,u)\right].
\end{gather*}
That is, if $x(t)$ is an optimal trajectory then there exists $p(t)$ such that the pair satisfy Hamilton's equations
\begin{equation}\label{eq:cont_hamilton}
    \dot{x} = \frac{\partial H}{\partial p}, \quad \dot{p} = -\frac{\partial H}{\partial x}.
\end{equation}
Moreover, this problem is a mixed boundary value problem with conditions
\begin{equation*}
    x(t_0) = x_0, \quad p(t_f) = d\varphi_{x(t_f)}.
\end{equation*}
\subsection{Lie-Poisson Reduction}
Let $H:T^*G\to\mathbb{R}$ be a left-invariant Hamiltonian, i.e.
\begin{equation*}
    H\left(g, p\right) = H\left( e, \left(\ell_g\right)^*p\right),
\end{equation*}
for all $g\in G$. Its restriction to the identity will be denoted by
\begin{gather*}
    \mathfrak{h}:\mathfrak{g}^*\to\mathbb{R}, \quad
    \mathfrak{h}(\mu) = H(e, \mu).
\end{gather*}
\begin{thm}[Lie-Poisson Reduction \cite{mech_symmetry}]
    Let $H:T^*G\to\mathbb{R}$ be a left-invariant Hamiltonian with restriction $\mathfrak{h}$. If $(g(t),p(t))$ is a solution to Hamilton's equations:
    \begin{equation}\label{eq:ham_equations}
        \dot{g} = \frac{\partial H}{\partial p}, \quad \dot{p} = -\frac{\partial H}{\partial g},
    \end{equation}
    then $\mu := (\ell_g)^*p$ satisfies the following
    \begin{equation}\label{eq:LP_equations}
            \dot{\mu} = \mathrm{ad}^*_{d\mathfrak{h}}\mu, \quad 
            \dot{g} = \left(\ell_g\right)_*d\mathfrak{h}.
    \end{equation}
\end{thm}
\subsection{Reduced Maximum Principle}
To obtain a left-invariant Hamiltonian, the control dynamics and Lagrangian must be left-invariant, i.e. (via a slight abuse of notation)
\begin{equation*}
    f(g,u) = \left(\ell_g\right)_*f(u), \quad L(g,u) = L(e,u).
\end{equation*}
In the case where the running cost, $L$, is independent of the group variable, we will denote it via
\begin{gather*}
    \mathfrak{L}:\mathcal{U}\to\mathbb{R}, \\
    \mathfrak{L}(u) = L(e,u).
\end{gather*}
It is straightforward to see that the optimal Hamiltonian is left-invariant and the resulting (restricted) Hamiltonian is
\begin{gather*}
    \mathfrak{h}:\mathfrak{g}^*\to\mathbb{R} \\
    \mathfrak{h}(\mu) = \min_u \, \left[ \langle \mu, f(u)\rangle + \mathfrak{L}(u)\right].
\end{gather*}

The maximum principle requires solving the mixed boundary value problem
\begin{equation*}
    g(0) = g_0, \quad p(T) = d\varphi_{g(T)},
\end{equation*}
subject to the dynamics \eqref{eq:ham_equations}. This problem can be translated to \eqref{eq:LP_equations} via $\mu(T) = \left(\ell_{g(T)}\right)^*p(T)$. An advantage of this approach is that it is well-suited for forward-backward techniques \cite{lenhart2007optimal}; the initial value and terminal value problems are now decoupled and can be solved independently.
\section{LEFT-INVARIANT HYBRID OPTIMAL CONTROL}\label{sec:LIHOC}
\subsection{Hybrid Maximum Principle}
For the cost function \eqref{eq:cost}, consider the optimization problem where the dynamics are now governed by the hybrid dynamics
\begin{equation*}
    \mathcal{HC}:\begin{cases}
        \dot{x} = f(x,u), & x\not\in\mathcal{S}, \\
        x^+ = \Delta(x^-), & x\in\mathcal{S}.
    \end{cases}
\end{equation*}
Suppose that $x(t)$ is an optimal (hybrid) arc. As a coddling assumption, suppose that $x$ only undergoes resets a finite number of times (and hence, the impact times are all uniformly separated). Between resets, the classical maximum principle still applies, i.e. the trajectory obeys the dynamics \eqref{eq:cont_hamilton}. At the reset time, the co-states jump according to the ``Hamiltonian jump condition'' \cite{pakniyat_2023}:
\begin{equation}\label{eq:ham_jump}
    \begin{split}
        p^+\circ\Delta_* - p^- &\in\mathrm{Ann}(T\mathcal{S}), \\
        H^+ - H^- &= 0,
    \end{split}
\end{equation}
where $\mathrm{Ann}$ is the annihilator:
\begin{equation*}
    \mathrm{Ann}(T\mathcal{S}) = \left\{ (x,p)\in T^*M : x\in\mathcal{S}, \ \langle p, T\mathcal{S}\rangle = 0\right\}.
\end{equation*}
\begin{remark}
    The assumption that the number of resets is finite is crucial to the above conditions and there exist quite simple control systems that do not obey this assumption \cite{zeno_hOC}.
\end{remark}
\begin{remark}\label{rmk:jump_assumption}
    One hopes that there exists a unique solution for $p^+$ in \eqref{eq:ham_jump}. In the case of mechanical impact systems (where $H$ is a mechanical Hamiltonian and $\Delta$ is the identity map), there always exists a unique (non-trivial) solution. Unfortunately, there is no reason a priori to assume that this is the case for an arbitrary control system - especially if $\Delta_*$ does not have full rank. This is not always the case as in many cases (e.g. legged locomotion), the reset map is a projection. Invertability of $\Delta_*$ is not strictly required, but the subsequent analysis becomes substantially more involved \cite{submersive_reset}.
\end{remark}
\subsection{Impact Lie-Poisson Reduction}
One fundamental difficulty in performing Lie-Poisson reduction to hybrid systems is that the $\mu$- and $g$-dynamics cannot be fully decoupled as in the continuous case. The reason for this is that resets cause a jump in $\mu$ but are triggered by the state $g$. This makes the full reduction from $2n$ down to $n$ dimensions impossible. Fortunately, it is possible to reduce from $2n$ down to $n+1$ where the additional dimension records whether or not a reset occurs.

\begin{thm}[Impact Lie-Poisson Reduction \cite{oprea2023study}]\label{thm:LPO}
    Consider a left-invariant impact system: $G$ a Lie group, $\Sigma = Kg_0$ a right coset with $K\trianglelefteq G$ a normal, co-dimension 1 closed Lie subgroup. Denote the natural projection map by $\pi:G\to \Sigma$ and the Lie algebra of $K$ by $\mathfrak{K}$.
    
    Let $H:T^*G\to \mathbb{R}$ be a mechanical left-invariant Hamiltonian and let $\mathfrak{h}$ be its restriction to the identity. Suppose $(g(t),p(t))$ follows the hybrid flow $\varphi_t^\mathcal{H}$ and let $\mu(t) = \left(\ell_{g(t)}\right)^*p(t)$. Let $\sigma:K\backslash G\to G$ be a local section, $q\in K\backslash G$, and $\delta\mu\in\mathrm{Ann}(\mathfrak{K})$ be the unique non-trivial solution to $\mathfrak{h}(\mu) = \mathfrak{h}(\mu + \delta \mu)$. Then $(q(t),\mu(t)$ follows the dynamics
    \begin{equation}\label{eq:cont_LPO}
        \begin{cases}
            \dot\mu = \mathrm{ad}^*_{d\mathfrak{h}}\mu, \\
            \dot{q} = d\pi_{\sigma(q)}\left(\ell_{\sigma(q)}\right)_*d\mathfrak{h}(\mu),
        \end{cases} \quad q\not\in [g_0],
    \end{equation}
    \begin{equation}
        \begin{cases}
            \mu \mapsto \mu + \delta\mu, \\
            q \mapsto q,
        \end{cases} \quad q\in [g_0].
    \end{equation}
\end{thm}
\begin{remark}
    The requirement that $\Sigma$ needs to be generated from a co-dimension 1 normal subgroup is quite restrictive. However, that condition is necessary for the reconstruction ($q$-dynamics) in \eqref{eq:cont_LPO} to be well-defined. The condition that $H$ be a mechanical Hamiltonian ensures that a unique (non-trivial) solution $\delta\mu$ for the co-state jumps exists. In more general contexts, the uniqueness/existence of solutions to $\mathfrak{h}(\mu)=\mathfrak{h}(\mu+\delta\mu)$ is not guaranteed.
\end{remark}
\subsection{Reduced Hybrid Maximum Principle}
Let $(G,\Sigma, \mathcal{U},f,\Delta)$ be a left-invariant hybrid control system with cost
\begin{equation}\label{eq:LI_cost}
    J = \int_{t_0}^{t_f} \, \mathfrak{L}(u) \, dt + \varphi\left( g(x_f)\right),
\end{equation}
where $\mathfrak{L}:\mathcal{U}\to\mathbb{R}$ is the running cost. The Hamiltonian arising from the maximum principle is left-invariant and has the form
\begin{equation}
    \mathfrak{h}(\mu) = \min_{u} \, \left[ \mathfrak{L}(u) + \langle \mu, f(u)\rangle\right].
\end{equation}
The resulting reduction results in the following dynamics.
\begin{thm}\label{thm:main_result}
    For a NLIHCS $(G, Kg_0, \mathcal{U}, f,r_{h_0})$ with invariant cost \eqref{eq:LI_cost}, regular optimal trajectories follow the continuous dynamics \eqref{eq:cont_LPO} and the modified jump conditions
    \begin{equation}\label{eq:opt_jump_LPO}
        \begin{cases}
            \mu \mapsto \mathrm{Ad}_{h_0}^*\mu + \delta\mu, \\
            q \mapsto q\cdot\pi\left(h_0\right),
        \end{cases} \quad q\in [g_0],
    \end{equation}
    where $\delta\mu\in\mathrm{Ann}(\mathfrak{K})$ such that $\mathfrak{h}(\mu^+) = \mathfrak{h}(\mu^-)$.
\end{thm}
\begin{proof}
    The continuous dynamics, \eqref{eq:cont_LPO}, follow from reducing Hamilton's flow from the maximum principle \eqref{eq:cont_hamilton}. The reset condition, \eqref{eq:opt_jump_LPO}, follows from reducing the Hamiltonian jump condition \eqref{eq:ham_jump} by Theorem \ref{thm:LPO} and noticing that
    \begin{equation*}
        \Delta_* = \left(r_{h_0}\right)_* \implies \left(\Delta_*\right)^{-1} = \left(r_{h_0^{-1}}\right)_*.
    \end{equation*}
    This results in the map
    \begin{equation*}
        p \mapsto \left(r_{{h_0}^{-1}}\right)^*p + \left(r_{{h_0}^{-1}}\right)^*\delta\mu.
    \end{equation*}
    Left-translating back to the identity yields
    \begin{equation*}
        \begin{split}
            \mu &\mapsto \left(\ell_{h_0}\right)^*\left(r_{h_0^{-1}}\right)^*\mu + \left(\ell_{h_0}\right)^*\left(r_{h_0^{-1}}\right)^*\delta\mu \\
            &\quad = \mathrm{Ad}_{h_0}^*\mu + \delta\mu,
        \end{split}
    \end{equation*}
    where $\mathrm{Ann}(\mathfrak{K})$ is closed under the co-Adjoint action as $K$ is a normal subgroup.
\end{proof}
\begin{remark}
    The jump condition \eqref{eq:opt_jump_LPO} tacitly makes the assumption that there exists a unique $\delta\mu$ such that 
    \begin{equation*}
        \mathfrak{h}\left( \mu\right) = \mathfrak{h}\left( \mathrm{Ad}_{h_0}^*\mu + \delta\mu \right).
    \end{equation*} 
    Unlike in the mechanical case, as discussed in Remark \ref{rmk:jump_assumption}, this is no longer guaranteed to be true. We will henceforth assume existence/uniqueness and relegate this question to future work.
\end{remark}
\begin{algorithm}
    \caption{Forward-Backward Solver for Hybrid Lie-Poisson Optimal Control}
    \begin{algorithmic}
        \State $g(T) \gets g_f^0$
        \State $\Delta g \gets \infty$
        \While {$\Delta g > tol$}
            \State $\mu_f \gets \left(\ell_{g(T)}\right)_*d\varphi_{g(T)}$
            \State $q_f \gets \pi(g(T))$
            \State \textbf{Solve} $\displaystyle \left[\begin{array}{c} \mu(t) \\ q(t) \end{array}\right]$ \textbf{backwards} : dynamics \eqref{eq:cont_LPO} and \eqref{eq:opt_jump_LPO}
            \State \textbf{Solve} $g(t)$ \textbf{forwards} : $\displaystyle \begin{cases}
                \dot{g} = \left(\ell_g\right)_*d\mathfrak{h}, & g\not\in\Sigma \\
                g^+ = \Delta(g^-), & g\in \Sigma
            \end{cases}$
            \State $g_f^{n+1} \gets g(T)$
            \State $\Delta g \gets \mathrm{dist}(g_f^{n}, g_f^{n+1})$
        
        \EndWhile
    \end{algorithmic}\label{alg:hybrid_LPO}
\end{algorithm}

\section{The Special Euclidean Group}\label{sec:SE2}
Consider the under-actuated control problem on $\mathrm{SE}_2$. The NLIHCS is a modified version of \eqref{eq:continuous_se2} and \eqref{eq:impact_se2} with continuous dynamics
\begin{equation}\label{eq:forward_se2_continuous}
    \begin{cases}
        \dot{x} = u\cos\theta, \\
        \dot{y} = -u\sin\theta, \\
        \dot{\theta} = \omega,
    \end{cases}\quad \theta \ne \dfrac{\pi}{2},
\end{equation}
and reset
\begin{equation}\label{eq:forward_se2_impact}
    \begin{cases}
        x \mapsto x, \\
        y \mapsto y - 1, \\
        \theta \mapsto \theta + \pi,
    \end{cases} \quad \theta = \dfrac{\pi}{2}.
\end{equation}
This reset map is equivalent to right-translation by the matrix
\begin{equation*}
    h_0 = \begin{pmatrix}
        -1 & 0 & 1 \\
        0 & -1 & 0 \\
        0 & 0 & 1
    \end{pmatrix}.
\end{equation*}
Let the cost be 
\begin{equation*}
    \begin{split}
        J &= \int_{t_0}^{t_f} \, \frac{1}{2}\left( u^2 + \omega^2 \right) \, dt + \varphi\left(x(t_f), y(t_f), \theta(t_f)\right).
    \end{split}
\end{equation*}
To reduce this system, we clarify that
\begin{gather*}
    f:\mathcal{U} = \mathbb{R}^2 \to \mathfrak{se}_2 \\
    f(v,\omega) = \begin{pmatrix}
        0 & \omega & u \\
        -\omega & 0 & 0 \\
        0 & 0 & 0
    \end{pmatrix},
\end{gather*}
and we denote the pairing $\mathfrak{se}_2^*\times\mathfrak{se}_2\to\mathbb{R}$ via
\begin{equation*}
    \left\langle \begin{pmatrix}
        \mu_x \\ \mu_y \\ \mu_\theta
    \end{pmatrix}, \begin{pmatrix}
        0 & \omega & u \\ -\omega & 0 & v \\ 0 & 0 & 0
    \end{pmatrix}\right\rangle = \mu_xu + \mu_yv + \mu_\theta\omega.
\end{equation*}
The (translated) Hamiltonian for this system is
\begin{equation*}
    \begin{split}
        \mathfrak{h}(\mu_x,\mu_y,\mu_\theta) &= \min_{u,\omega} \, \left[ \mu_xu + \mu_\theta\omega + \frac{1}{2}\left(u^2+\omega^2\right)\right] \\
        &= -\frac{1}{2}\left( \mu_x^2 + \mu_\theta^2\right),
    \end{split}
\end{equation*}
with controls $u = -\mu_x$ and $\omega = -\mu_\theta$. The Lie-Poisson equations on $\mathfrak{se}_2^*$ are
\begin{equation*}
    \begin{split}
        \dot{\mu}_x = -\mu_y\frac{\partial\mathfrak{h}}{\partial\mu_\theta}, \quad
        \dot{\mu}_y = \mu_x\frac{\partial\mathfrak{h}}{\partial \mu_\theta}, \quad
        \dot{\mu}_\theta = \mu_y\frac{\partial\mathfrak{h}}{\partial\mu_x} - \mu_x\frac{\partial\mathfrak{h}}{\partial \mu_y}
    \end{split}
\end{equation*}
which results in the continuous dynamics
\begin{equation}\label{eq:se2_continuousLP}
    \begin{array}{ll}
        \dot{\mu}_x = -\mu_y\mu_\theta, & \dot{\mu}_y = \mu_x\mu_\theta, \\
        \dot{\mu}_\theta = -\mu_x\mu_y, & \dot{\theta} = -\mu_\theta.
    \end{array}
\end{equation}
The last equation is the reduced $q$-dynamics on $K\backslash G\cong \mathbb{S}^1$. To compute the jump in the co-states, we notice that
\begin{equation*}
    \mathrm{Ad}_{h_0}^*\begin{pmatrix}
        \mu_x \\ \mu_y \\ \mu_\theta
    \end{pmatrix} = \begin{pmatrix}
        -\mu_x \\ -\mu_y \\ \mu_y - \mu_\theta
    \end{pmatrix}.
\end{equation*}
The modified jump condition has the form
\begin{equation*}
    \begin{pmatrix}
        \mu_x \\ \mu_y \\ \mu_\theta
    \end{pmatrix} \mapsto \begin{pmatrix}
        -\mu_x \\ -\mu_y \\ \mu_y-\mu_\theta
    \end{pmatrix} + \begin{pmatrix}
        0 \\ 0 \\ \varepsilon
    \end{pmatrix}.
\end{equation*}
Enforcing energy conservation actually yields two solutions:
\begin{equation}\label{eq:se2_resetLP}
    \begin{split}
        \mu_x^+ &= -\mu_x, \qquad \quad
        \mu_y^+ = -\mu_y, \\
        \mu_\theta^+ &= \pm \mu_\theta^-, \qquad\quad \theta^+ = \theta^-+\pi.
    \end{split}
\end{equation}
This system can be reduced further as there exists a Casimir
\begin{equation*}
    C = \mu_x^2 + \mu_y^2,
\end{equation*}
which is invariant under \underline{both} \eqref{eq:se2_continuousLP} and \eqref{eq:se2_resetLP}. Incorporating this conserved quantity allows the optimal control problem to be encoded in the following 3-dimensional hybrid system
\begin{equation}\label{eq:backwards_se2_continuous}
    \theta\ne\frac{\pi}{2}:\begin{cases}
        \dot{\alpha} = \mu_\theta, \\
        \dot{\mu}_\theta = -C^2\sin\alpha\cos\alpha, \\
        \dot{\theta} = -\mu_\theta
    \end{cases}
\end{equation}
\begin{equation}\label{eq:backwards_se2_impact}
    \theta = \frac{\pi}{2}: \begin{cases}
        \alpha^+ = \alpha^- + \pi, \\
        \mu_\theta^+ = \pm \mu_\theta^-, \\
        \theta^+ = \theta^- + \pi.
    \end{cases}
\end{equation}
The new variable, $\alpha$, comes from the Casimir via
\begin{equation*}
    \mu_x = C\cos\alpha,\quad \mu_y = C\sin\alpha.
\end{equation*}
Solving the optimal control problem via Algorithm \ref{alg:hybrid_LPO} involves solving \eqref{eq:forward_se2_continuous} and \eqref{eq:forward_se2_impact} forwards, and \eqref{eq:backwards_se2_continuous} and \eqref{eq:backwards_se2_impact} backwards.

Notice that $\alpha+\theta \equiv D$ is a hybrid constant of motion (modulo $2\pi$). The resulting $(\theta,\mu_\theta)$-dynamics are (impact) Hamiltonian with 
\begin{equation*}
    H(\theta,\mu_\theta) = -\frac{1}{2}\mu_\theta^2 + \frac{C^2}{4}\cos\left(2D-2\theta\right).
\end{equation*}
Applying these conserved quantities to the full system yields the continuous dynamics
\begin{equation}\label{eq:decon_recon_cont}
    \begin{split}
        \dot{x} &= -C\cos(D-\theta)\cos\theta, \\
        \dot{y} &= C\cos(D-\theta)\sin\theta, \\
        \dot{\theta} &= -\mu_\theta, \\
        \dot{\mu}_\theta &= -C^2\sin(D-\theta)\cos(D-\theta),
    \end{split}
\end{equation}
along with the reset equation (when $\theta = \pi/2$)
\begin{equation}\label{eq:decon_recon_reset}
    \begin{split}
        x \mapsto x, \qquad\qquad  &\theta \mapsto \theta + \pi, \\
        y \mapsto y - 1, \qquad \ &\mu_\theta \mapsto \pm\mu_\theta.\\
    \end{split}
\end{equation}
A sample trajectory of the dynamics \eqref{eq:decon_recon_cont} and \eqref{eq:decon_recon_reset}, along with the running cost, are shown in Fig. \ref{fig:state_trajectory} and \ref{fig:running_cost}.
\begin{figure}
    \centering
    \includegraphics[width=\columnwidth]{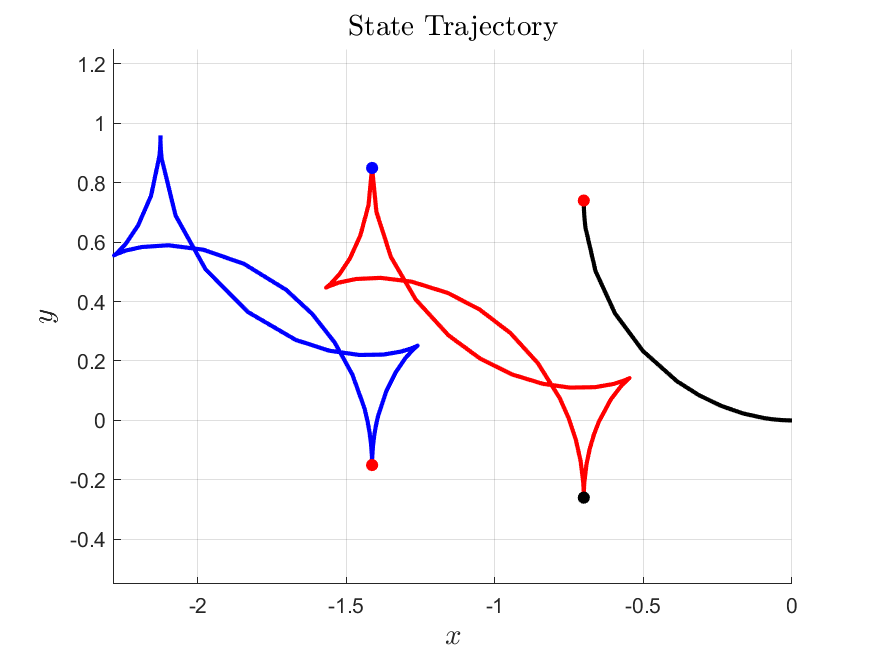}
    \caption{A trajectory with initial conditions $(x_0,y_0,\theta_0) = (0,0,0)$. The auxiliary parameters are $C=D=1$ and $\mu_\theta(0) = -1$. The black arc is the initial trajectory which maps to the red trajectories after the first reset. The two red arcs indicate following both lifted reset maps \eqref{eq:decon_recon_reset}. Notice that the next reset (red to blue) is independent on which reset map was chosen.}
    \label{fig:state_trajectory}
\end{figure}
\begin{figure}
    \centering
    \includegraphics[width=\columnwidth]{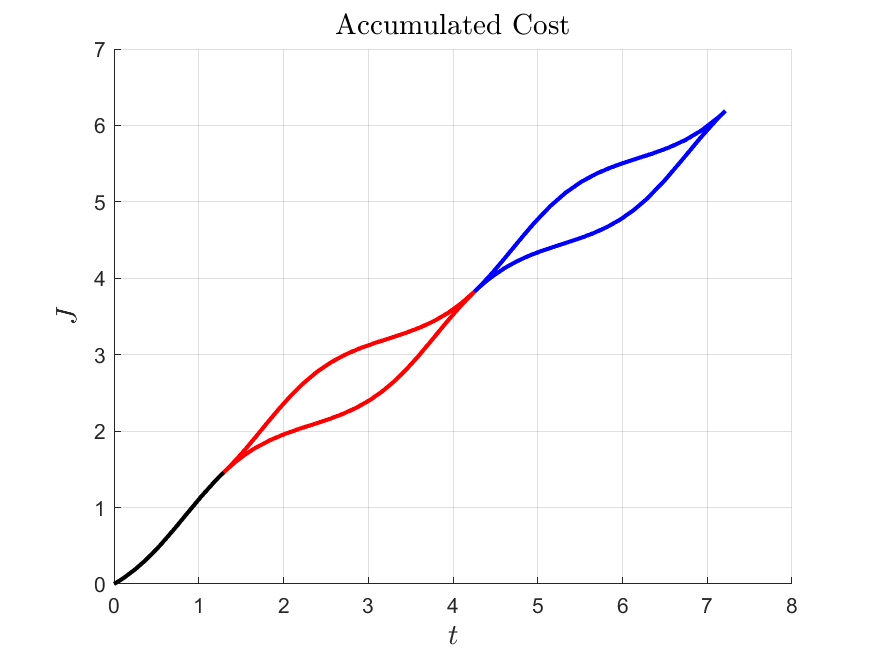}
    \caption{A plot of the running cost against time corresponding to the trajectories in Fig. \ref{fig:state_trajectory}. This cost is multi-valued as there are two choices for \eqref{eq:decon_recon_reset}. Notice that at resets, both options result in the same total cost.}
    \label{fig:running_cost}
\end{figure}
\section{CONCLUSION}\label{sec:conclusion}
This work presented the analogous version of Lie-Poisson reduction to left-invariant optimal control problems where the underlying dynamics are hybrid. Unlike the classical Lie-Poisson reduction where the resulting dynamics are $n$-dimensional, the hybrid version can only reduce the dimension to $n+1$. Algorithm \ref{alg:hybrid_LPO} was presented as an approach to determine the optimal trajectories of such systems. There are two main avenues for future work along this subject.

The performance of Algorithm \ref{alg:hybrid_LPO} needs studied. This algorithm is an iterative approach to determine the optimal hybrid trajectory. Although the optimal trajectory is a fixed-point of this iterative scheme, the stability and convergence properties are unknown. To be of use in practical applications, criteria need to be determined to ensure that the iterations converge exponentially to the true solution. Alternatively, if the iterations are unstable, there should be a way to modify the iterations to stabilize the scheme.

The other immediate focus in the future is the prohibitive restriction on the guard that it be a coset of a normal subgroup. For a given (arbitrary) Lie group, just possessing a co-dimension 1, closed, normal subgroup is by no means guaranteed; $\mathrm{SO}_3$ and $\mathrm{SE}_3$ do not have such a subgroup. If this subgroup does not exist, then the proposed reduction does not work. As the normality of $K$ is required to reduce the dimension to $n+1$, it would be interesting to determine whether or not relaxing this requirement would lead to a reduced system of dimension $n+k$ for some $1<k<n$.

\addtolength{\textheight}{-12cm}   






\bibliographystyle{IEEEtran}
\bibliography{ref.bib}

\end{document}